\documentclass[12pt]{amsart}

\usepackage{latexsym}
\usepackage{amsfonts}
\usepackage{amssymb}
\usepackage{rotating}
\usepackage{enumerate}
\usepackage{amsmath}
\usepackage{amscd}
\usepackage{amsthm}
\usepackage{fullpage}

\def\ints{{\mathbb Z}}

\def\proj{{\mathbb P}}
\def\aff{{\mathbb A}}
\def\FF{{\mathbb F}}

\def\qed{\hfill $\Box$}

\def\Frac{{\text{Frac}}}

\def\Gal{{\text{Gal}}}

\def\mc#1{\mathcal{#1}}
\def\mf#1{\mathfrak{#1}}
\def\ol#1{\overline{#1}}

\newtheorem{theorem}[equation]{Theorem}
\newtheorem{prop}[equation]{Proposition}
\newtheorem{lemma}[equation]{Lemma}
\newtheorem{corollary}[equation]{Corollary}
\newtheorem{predefinition}[equation]{Definition}

\newtheorem{preremark}[equation]{Remark}
\newenvironment{remark}{\begin{preremark}\rm}{\end{preremark}}
\newtheorem{preproblem}[equation]{Problem}
\newenvironment{prob}{\begin{preproblem}\rm}{\end{preproblem}}
\newtheorem{preconstruction}[equation]{Construction}

\newtheorem{prenotation}[equation]{Notation}

\newtheorem{preexample}[equation]{Example}

\newtheorem{preclaim}[equation]{Claim}

\newtheorem{prequestion}[equation]{Question}
\newenvironment{question}{\begin{prequestion}\rm}{\end{prequestion}}

\numberwithin{equation}{section}
\begin{document}

\title{The local lifting problem for $A_4$}

\author{Andrew Obus}
\address{University of Virginia, 141 Cabell Drive, Charlottesville, VA
22904}
\email{andrewobus@gmail.com}
\thanks{The author was supported by NSF FRG grant DMS-1265290.  He
  thanks Florian Pop, David Harbater and the referee for useful comments,
  and Adam Levine for inadvertently suggesting Corollary \ref{Cglobal}
  while making a joke about paper sizes.}
\let\thefootnote\relax
\date{\today}
\begin{abstract}
We solve the local lifting problem for the alternating group $A_4$,
thus showing that it is a \emph{local Oort group}.  Specifically, if $k$ is an algebraically closed field of characteristic
$2$, we prove that every $A_4$-extension of $k[[s]]$ lifts to characteristic
zero.  
\end{abstract}
\maketitle

\section{Introduction}\label{Sintro}

This paper concerns the \emph{local lifting problem} about lifting
Galois extensions of power series rings from characteristic $p$ to
characteristic zero:

\begin{prob}(The local lifting problem)\label{llp}
  Let $k$ be an algebraically closed field of characteristic $p$ and $G$ a finite group.  
  Let $k[[z]]/k[[s]]$ be a $G$-Galois extension (that is, $G$ acts
  on $k[[z]]$ by $k$-automorphisms with fixed ring $k[[s]]$).  Does this
  extension lift to characteristic zero?  That is,
 does there exist a DVR $R$ of characteristic zero with residue field $k$ and a
  $G$-Galois extension $R[[Z]]/R[[S]]$ that reduces to
  $k[[z]]/k[[s]]$?  
\end{prob}

We will refer to a $G$-Galois extension $k[[z]]/k[[s]]$ as a
\emph{local $G$-extension}.
Basic ramification theory shows that any group $G$ that occurs
as the Galois group of a local extension is of the form
$P \rtimes \ints/m$, with $P$ a $p$-group and $p \nmid m$.
In \cite{CGH:ll}, Chinburg, Guralnick, and Harbater
ask, given a prime $p$, for which groups $G$ (of the form $P \rtimes \ints/m$) is it true that
all local $G$-actions (over all algebraically closed fields of characteristic $p$) lift to
characteristic zero? Such a group is called a \emph{local Oort group}
(for $p$).  Due to various obstructions (The \emph{Bertin obstruction}
of \cite{Be:ol}, the \emph{KGB
  obstruction} of \cite{CGH:ll}, and the \emph{Hurwitz tree
  obstruction} of \cite{BW:ac}), the list of possible local Oort groups is
quite limited. In particular, due to \cite[Theorem 1.2]{CGH:ll} and
\cite{BW:ac}, if a group $G$ is a local Oort group for $p$, then $G$ is
either cyclic, dihedral of order $2p^n$, or the alternating group $A_4$
($p=2$).  Cyclic groups are known to be local Oort --- this is the
so-called \emph{Oort conjecture}, proven by Obus-Wewers and Pop in
\cite{OW:ce}, \cite{Po:oc}.  Dihedral groups of order $2p$ are known
to be local Oort for $p$ odd due to Bouw-Wewers (\cite{BW:ll}) and for
$p = 2$ due to Pagot (\cite{Pa:rc}).
The group $D_9$ is local Oort by \cite{Ob:go}.  Our main theorem is: 

\begin{theorem}\label{Tmain}
If $k$ is an algebraically closed field of characteristic $2$, then every $A_4$-extension of $k[[s]]$ lifts to characteristic zero.
That is, the group $A_4$ is a local Oort group for $p = 2$.
\end{theorem}

This result was announced by Bouw (see the beginning of \cite{BW:ll}), but the proof has
not been written down.  Our proof uses a simple idea that
avoids the ``Hurwitz tree'' machinery of \cite{BW:ll}.
Namely, one first classifies local $A_4$-extensions by what we call their ``break'' (this
is a jump in the higher ramification filtration).  One then uses the
following
strategy of Pop (\cite{Po:oc}), sometimes known as the ``Mumford method'': First, make an equicharacteristic deformation
of a local $A_4$-extension such that, generically, the break of the extension
goes down.  If one can lift the local extensions arising from
the generic fiber of this deformation, Pop's work shows that one can lift the original
extension.  On the other hand, we show explicitly that local $A_4$-extensions with small
breaks lift.  An induction finishes the proof.

We remark that Florian Pop has his own similar proof of Theorem \ref{Tmain}, which was
communicated to the author after the first draft of this paper was
written (see Remark \ref{Rpop}).

The main motivation for the local lifting problem is the following
\emph{global lifting problem}, about deformation of curves with an action of
a finite group (or equivalently, deformation of Galois branched covers
of curves).

\begin{prob}(The global lifting problem)\label{glp}
Let $X/k$ be a smooth, connected, projective curve over an
algebraically closed field of characteristic $p$.  Suppose a finite
group $\Gamma$ acts on $X$.  Does $(X, \Gamma)$ lift to characteristic
zero?  That is, does there exist a DVR $R$ of characteristic
zero with residue field $k$ and a relative projective curve $X_R/R$
with $\Gamma$-action such that $X_R$, along with its $\Gamma$-action,
reduces to $X$?
\end{prob}

It is a major result of Grothendieck (\cite[XIII, Corollaire
2.12]{SGA1}) that the global lifting problem can be solved whenever $\Gamma$
acts with \emph{tame} (prime-to-$p$) inertia groups, and $R$ can be taken to
be the Witt ring $W(k)$.  
More generally, the \emph{local-global
  principle} states that $(X,
\Gamma)$ lifts to characteristic zero over a complete DVR
$R$ if and only if the
local lifting problem holds (over $R$) for each point of $X$ with nontrivial
stabilizer in $\Gamma$.  Specifically, if $x$ is such a point, then its
complete local ring is isomorphic to $k[[z]]$. The stabilizer $I_x
\subseteq \Gamma$ acts on $k[[z]]$ by $k$-automorphisms, and we check the local lifting
problem for the local $I_x$-extension $k[[z]]/k[[z]]^{I_x}$.  Thus,
the global lifting problem is reduced to the local lifting problem.
Proofs of the local-global principle have been given by
Bertin-M\'{e}zard (\cite{BM:df}), Green-Matignon (\cite{GM:lg}), and Garuti (\cite{Ga:pr}).  

One consequence of the local-global principle and Theorem \ref{Tmain} is the
following:

\begin{corollary}\label{Cglobal}
The groups $A_4$ and $A_5$ are so-called \emph{Oort groups} for every prime.
That is, if $\Gamma \in \{A_4, A_5\}$ acts on a smooth projective curve $X$ over an algebraically closed field of
positive characteristic $p$, then $(X, \Gamma)$ lifts to characteristic zero.
\end{corollary}

\begin{proof}
By the local-global principle (see also \cite[Theorem 2.4]{CGH:og}), it suffices to
show that every cyclic-by-$p$ subgroup of $A_4$ or $A_5$ is a local
Oort group for $p$. The only subgroups of $A_4$ of this form for any
$p$ are isomorphic to the trivial group, $\ints/2$, $\ints/2 \times \ints/2$, $\ints/3$,
or $A_4$.  The subgroups of $A_5$ of this form are isomorphic to the
trivial group, $\ints/2$, $\ints/2 \times \ints/2$, $\ints/3$,
$\ints/5$, $D_3$, $A_4$, and $D_5$.  All these are local Oort groups
for the relevant primes,
as has been noted above.
\end{proof}

\subsection{Coventions/Notation}
 
Throughout, $k$ is an algebraically closed field of characteristic
$2$.  The ring $R$ is a large enough complete discrete valuation ring
of characteristic zero with residue field $k$, maximal ideal $\mf{m}$,
and uniformizer $\pi$.  We normalize the
valuation $v$ on $R$ so that $v(2) = 1$.  In any polynomial or power
series ring with coefficients in $R$, the expression $o(x)$ for $x \in
R$ means a polynomial or power series with coefficients in $x\mf{m}$.

The ring $k[[t]]$ is always a $\ints/3$-extension of $k[[s]]$
with $t^3 = s$. Likewise, $R[[T]]$ is always a $\ints/3$-extension of
$R[[S]]$ with $T^3 = S$.  If $k[[z]]/k[[s]]$ is an extension, it is
always assumed to contain $k[[t]]$.  Our convention for variables is that
lowercase letters represent the reduction of capital letters from
characteristic $0$ to characteristic $2$ (e.g., $t$ is the reduction
of $T$).

We write $\zeta_3$ for a primitive $3$rd root of unity in any ring.

\section{$A_4$-extensions}\label{SA4}
We start with the basic structure theory of $A_4$-extensions.
\subsection{$A_4$-extensions in characteristic $2$}\label{SA4char2}
\begin{lemma}\label{LA4}
Let $K \subseteq L \subseteq M$ be a tower of field extensions of characteristic
$2$ such that $L/K$ is $\ints/3$-Galois and $\Gal(M/L)$ is
$\ints/2$-Galois.  Let $\sigma$ be a generator of $\Gal(L/K)$.  For
$\ell \in L$, let $\ol{\ell}$ denote the image of $\ell$ in
$L/(F-1)L$, where $F$ is Frobenius.  Suppose $M
\cong L[x]/(x^2 - x - a)$, and let $d$ be the dimension of the $\FF_2$-vector space
generated by $\ol{a}, \ol{\sigma(a)}$, and $\ol{\sigma^2(a)}$.  If $N$ is the Galois closure of $M$ over
$L$, then $\Gal(N/K)$ can be expressed as a semi-direct product $\cong (\ints/2)^d \rtimes \ints/3$.
\end{lemma}

\begin{proof}
By the Schur-Zassenhaus theorem, it is enough to prove that $\Gal(N/L)
\cong (\ints/2)^d$.  But $N/L$ is clearly generated by Artin-Schreier
roots of $a$, $\sigma(a)$, and $\sigma^2(a)$.  Thus the result follows
from Artin-Schreier theory.
\end{proof}

\begin{corollary}\label{CA4}
If $d = 2$ in Lemma \ref{LA4}, then $\Gal(N/K) \cong A_4$.
\end{corollary}

\begin{proof}
The group $\Gal(N/K)$ must be a semi-direct
product $(\ints/2)^2 \rtimes \ints/3$ that is nonabelian (as
there exists a non-Galois subextension).  The only such group is $A_4$.
\end{proof}

If $K = k((s))$ in Lemma \ref{LA4} above, then after a change of
variable, we may assume that $L = k((t))$ with $t^3 = s$.  Then,
it is easy to see that an Artin-Schreier representative $a$ of $M/L$ may be chosen uniquely
such that $a \in t^{-1}k[t^{-1}]$ and $a$ has only odd-degree terms.
We say that such an $a$ is in \emph{standard form}.
In this case, a standard exercise shows that the break in the higher
ramification filtration of $M/L$ (i.e., the largest $i$ such that
the higher ramification group $G_i$ is nontrivial)
occurs at $\deg(a)$, thought of as a polynomial in $t^{-1}$.  

\begin{corollary}\label{CA42}
Suppose $K = k((s))$ and $L = k((t))$.  Suppose $a \in
t^{-1}k[t^{-1}] \subseteq L$ is in standard form.  Using the notation
of Lemma \ref{LA4}, we have $\Gal(N/K) \cong A_4$ if and only if $a$
has no nonzero terms of degree divisible by $3$.
\end{corollary}

\begin{proof}
Since linear combinations of elements of $L$ in standard form are also
in standard form, Lemma \ref{LA4} and Corollary \ref{CA4} imply that
$\Gal(N/K) = A_4$ if and only if the $\FF_2$-subspace $V$ of $L$ generated by $a$, $\sigma(a)$, and
$\sigma^2(a)$ has dimension $2$.  If $a$ has no nonzero terms of degree
divisible by $3$, then $a + \sigma(a) + \sigma^2(a) = 0$ is the only
$\FF_2$-linear relation that holds among the conjugates of $a$, so
$\dim V = 2$ (note
that $a \neq 0$ since it is an Artin-Schreier representative of
$M/L$).  Conversely, if $a$ has a nonzero term of degree divisible by
$3$, then either no $\FF_2$-linear relation holds, or $a \in k((s))$ (in which
case $a = \sigma(a) = \sigma^2(a)$).  In either case, $\dim V \neq 2$. 
\end{proof}

If $d = 2$ in the context of Lemma \ref{LA4}, we say that $a \in L$
\emph{gives rise to} the $A_4$-extension $N/K$.  By abuse of notation,
if $K \cong k((s))$, we say that the \emph{break} of $N/K$ is the
ramification break of $M/L$.  This is the same as the unique
ramification break of $N/L$ in either the upper or lower numbering.  Furthermore, if $K = k((s))$ and $N =
k((z))$, we also say that $a$ gives rise to the extension $k[[z]]/k[[s]]$.

\begin{prop}\label{P1or5}
If $K = k((s))$ and $N/K$ is an $A_4$-extension with break $\nu$,
then $\nu \equiv 1 \text{ or } 5 \pmod{6}$.
\end{prop}
\begin{proof}
If $a$ gives rise to $N/K$ and is in standard form, we know that $\nu$
is the degree of $a$ in $t^{-1}$.  This must
be odd, and by Corollary \ref{CA42}, it cannot be divisible by $3$.
\end{proof}

\subsection{$A_4$-extensions in characteristic zero}\label{SA4char0}
The story in characteristic zero (or odd characteristic) is completely analogous. We state the result for
reference and omit the proof, which is the same as in
\S\ref{SA4char2} with Kummer theory substituted for Artin-Schreier theory.

\begin{prop}\label{PA4char0}
Let $K \subseteq L \subseteq M$ be a tower of separable field extensions of characteristic
$\neq 2$ such that $L/K$ is $\ints/3$-Galois and $\Gal(M/L)$ is
$\ints/2$-Galois. Let $\sigma$ be a generator of $\Gal(L/K)$.  For
$\ell \in L^{\times}$, let $\ol{\ell}$ denote the image of $\ell$ in
$L^{\times}/(L^{\times})^2$.  Suppose $M
\cong L[x]/(x^2 - a)$, 
and let $d$ be the dimension of the $\FF_2$-subspace of $L^\times/(L^\times)^2$
generated by $\ol{a}, \ol{\sigma(a)}$, and $\ol{\sigma^2(a)}$.  If $N$ is the Galois closure of $M$ over
$L$, then $\Gal(N/K)$ can be expressed as a semi-direct product $\cong (\ints/2)^d \rtimes \ints/3$.
In particular, if $d = 2$, then $\Gal(N/K) \cong A_4$.
\end{prop}

In the context of Proposition \ref{PA4char0}, we again say that $a \in
L$ \emph{gives rise} to $N/K$.



\section{Characteristic $2$ deformations}\label{Sdef}

For this section, let $K$, $L$, $M$, $N$ be as in \S\ref{SA4char2},
with $K = k((s))$ and $L = k((t))$.  Let $N = k((z))$.  Suppose
$\Gal(N/K) \cong A_4$, and $N/K$ is given rise to by $a \in
t^{-1}k[t^{-1}]$ in standard form.  Let $\nu$ be the break
of $N/K$.  Our goal is to prove the following proposition.

\begin{prop}\label{Pmaindef}
Suppose that $\nu > 6$ and all $A_4$-extensions $N'/K$ with break $\leq \nu -
6$ lift to characteristic zero.  Then $N/K$ lifts to
characteristic zero.
\end{prop}

Our proof follows an idea of Pop (\cite{Po:oc}).  As in \cite{Po:oc}
and \cite{Ob:go}, we make a deformation in characteristic
$2$ so that the generic fiber has ``milder'' ramification than the
special fiber.

\begin{prop}\label{Pcharpoort}
Let $\mc{A} = k[[\varpi, s]] \supseteq k[[s]]$, and let $\mc{K} = \Frac(\mc{A})$.  
There exists an $A_4$-extension $\mc{N}/\mc{K}$, with $\mc{N} \supseteq N$, having the following properties:

\begin{enumerate}
\item The unique $\ints/3$-subextension $\mc{L}/\mc{K}$ of $\mc{N}/\mc{K}$ 
is given by $\mc{L} = \mc{K}[t] \subseteq \mc{N}$.
\item If $\mc{C}$ is the integral closure of $\mc{A}$ in $\mc{N}$, we have $\mc{C} \cong k[[\varpi, z]]$. 
In particular, $(\mc{C}/(\varpi))/(\mc{A}/(\varpi))$ is $A_4$-isomorphic to the original extension $k[[z]]/k[[s]]$.
\item Let $\mc{B} = \mc{A}[t] \subseteq \mc{L}$.  Let $\mc{R} = \mc{A}[\varpi^{-1}]$, let $\mc{S} = \mc{B}[\varpi^{-1}]$, 
and let $\mc{T} = \mc{C}[\varpi^{-1}]$.  
Then $\mc{T}/\mc{R}$ is an $A_4$-extension of Dedekind rings, branched at $2$ maximal ideals.  Above the ideal
$(s)$, the inertia group is $A_4$, and the break is $\nu - 6$.  The other branched ideal has
inertia group $\ints/2 \times \ints/2$, unique ramification break $1$, and can be chosen to
be of the form $(s - \mu^3)$, where $\mu \in \varpi^2 k[[\varpi^2]]
\backslash \{0\}$ is arbitrary.      
\end{enumerate}
\end{prop}

\begin{proof}
Define $\mc{L}$ by adjoining $t$ to $\mc{K}$.
We proceed by deforming $a$ to an element of $\mc{L}$.  Let $\mu \in \varpi^2 k[[\varpi^2]]
\backslash \{0\}$.  Let $a' = a/t^{-6} = a/s^{-2}$, and deform $a$
to the element $\tilde{a} := a'(s - \mu^3)^{-2} =
a'\prod_{\alpha=1}^3(\zeta_3^{\alpha}t - \mu)^{-2} \in \mc{B}_{(\varpi)}
\subseteq \mc{L}$.  Note that $\tilde{a}$ reduces to $a \pmod{\varpi}$.  
Observe also that $\Gal(\mc{L}/\mc{K}) \cong \ints/3$, and the
$\FF_2$-vector space generated by the images
of $a'$ (and thus $\tilde{a}$) under this Galois action has dimension
$2$.  By Corollary \ref{CA4}, $\tilde{a} \in \mc{L}$ gives rise to an
$A_4$-extension $\mc{N}/\mc{K}$.  We claim that this is the extension
we seek.

Property (1) is obvious.  To show property (3), first note that
$\mc{S}/\mc{R}$ is branched exactly above the ideal $(s)$.  The
$\ints/2$-subextensions of $\mc{N}/\mc{L}$ are the Artin-Schreier extensions
corresponding to $\tilde{a}$, $\sigma(\tilde{a})$, and
$\sigma^2(\tilde{a})$, where $\sigma$ generates
$\Gal(\mc{L}/\mc{K})$.  Each of these is ramified at most above the ideals $(t)$
and $(\zeta_3^{\alpha}t - \mu)$, for $\alpha \in \{1, 2, 3\}$.  We will
see in the next paragraph that all of these ideals ramify in each $\ints/2$-subextension.
Thus the ramification groups of $\mc{T}/\mc{S}$ above these ideals are all $\ints/2 \times \ints/2$.  Since the
three $\ints/2$-subextensions are Galois conjugate over $\mc{K}$,
there can only be one higher ramification jump for each ideal, and it
is determined, say, by the Artin-Schreier subextension corresponding to $\tilde{a}$.

To determine the ramification, we consider the Artin-Schreier extension of the complete discrete valuation field $k((\varpi))((t))$ (resp.\
$k((\varpi))((\zeta_3^{\alpha}t - \mu))$ for $\alpha \in \{1, 2, 3\}$)
given by $\tilde{a}$.  
Since $t$ is a unit in $k((\varpi))[[\zeta_3^{\alpha}t - \mu]]$ for
any $\alpha$ and $\zeta_3^{\alpha}t - \mu$ is a unit in
$k((\varpi))[[t]]$ and in
$k((\varpi))[[\zeta_3^{\alpha'}t - \mu]]$ for any $\alpha' \neq
\alpha$ in $\{1, 2, 3\}$, the degree of the pole of $\tilde{a}$ with respect to $t$ (resp.\
$\zeta_3^{\alpha}t - \mu$) is $\nu - 6$ (resp.\ $2$).  Since $\nu - 6$
is odd, we have that the Artin-Schreier extension of
$k((\varpi))((t))$ given by $\tilde{a}$ ramifies and has ramification break $\nu -
6$.  To calculate the ramification break for the corresponding
extension of $k((\varpi))((\zeta_3^{\alpha}t - \mu))$, we assume $\alpha =
3$ for simplicity and we write
$\tilde{a}$ as a Laurent series in $(t - \mu)$.  Note that
$\tilde{a} = t^{-1}(t - \mu)^{-2}x^2$ for some $x \in k((\varpi))[[t -
\mu]]^{\times}$, and that 
$$t^{-1} = \mu^{-1} + \mu^{-2}(t-\mu) + \text{higher order terms in
}(t - \mu).$$  So 
$$\tilde{a} = c\mu^{-1}(t - \mu)^{-2} + c\mu^{-2}(t - \mu)^{-1} +
\theta,$$ 
where $\theta \in k((\varpi))[[t - \mu]]$ and $c \in k((\varpi))$ is
the ``constant'' term of $x^2$ (in fact, it is easy to see that $c \in
k((\mu^2)) = k((\varpi^4))$).  Let $b =
\sqrt{c\mu^{-1}}(t-\mu)^{-1}$.  After replacing
$\tilde{a}$ with $\tilde{a} + b^2 - b$, which does not change the
Artin-Schreier extension, we see that $\tilde{a}$ has a simple
pole (since $c \neq \mu^3$, the principal part
does not vanish).  So this extension ramifies with ramification break $1$.  This
shows property (3).

For property (2), it suffices by \cite[I, Theorem 3.4]{GM:lg} to show that the total degree of the different of $\mc{T}/\mc{R}$ is
equal to the degree of the different of $N/K$. Clearly, we
may replace $\mc{R}$ by $\mc{S}$ and $K$ by $L$.  Call these degrees
$\delta_{\mc{T}/\mc{S}}$ and $\delta_{N/L}$, respectively.

Since the ramification break of $M/L$ is $\nu$, and $N/L$ is the
compositum of Galois conjugates of $M/L$, we have that $N/L$ has $\nu$
as its single ramification break in the upper numbering, and all
nontrivial higher ramification groups of $N/L$ have order $4$.
Using Serre's different formula (\cite[IV, Proposition 4]{Se:lf}), we
obtain $\delta_{N/L} = 3(\nu + 1)$. 

For $\delta_{\mc{T}/\mc{S}}$, we add up the contributions from the
different branched ideals separately.  For the ideal $(t)$, argung as
in the previous paragraph, we have a
$\ints/2 \times \ints/2$-extension with single ramification break $\nu
- 6$.  This gives a contribution of $3(\nu - 5)$ to
$\delta_{\mc{T}/\mc{S}}$.  For each of the branched ideals $(\zeta^{\alpha}t -
\mu)$ ($\alpha \in \{1, 2, 3\}$), we have ramification group $\ints/2
\times \ints/2$ with ramification break $1$.  Using Serre's different formula again,
we get a contribution of $3 \cdot 3 \cdot 2 =
18$ to $\delta_{\mc{T}/\mc{S}}$.  Thus 
$\delta_{\mc{T}/\mc{S}} = 3(\nu - 5) + 18 = \delta_{N/L}$ and we are done.

\end{proof}

We omit the proof of the following proposition, which follows from Proposition \ref{Pcharpoort} exactly as \cite[Theorem 3.6]{Po:oc} follows from
\cite[Key Lemma 3.2]{Po:oc}.

\begin{prop}\label{Pcharpglobal}
Let $Y \to W$ be a branched $A_4$-cover of projective smooth
$k$-curves.  Suppose that the local inertia at each totally ramified point
is an extension $k[[z]]/k[[s]]$ having break
$\leq \nu$ and given rise to by an Artin-Schreier generator in
standard form divisible by $t^{-6}$ in $k[t^{-1}]$.  Set $\mc{W} = W \times_k k[[\varpi]]$.  Then there is an $A_4$-cover of 
projective smooth $k[[\varpi]]$-curves $\mc{Y} \to \mc{W}$ with
special fiber $Y \to W$ such that the totally ramified points
on the generic fiber $\mc{Y}_{\eta} \to \mc{W}_{\eta}$ have breaks $\leq \nu - 6$.
\end{prop}

Before we prove Proposition \ref{Pmaindef}, we recall
\emph{Harbater-Katz-Gabber} covers (or \emph{HKG-covers}) from
\cite{Ka:lg}.  Let $G \cong P \rtimes \ints/m$, with $P$ a $p$-group
and $p \nmid m$.  If $k[[z]]/k[[s]]$ is a local
$G$-extension,
then the associated HKG-cover is the unique branched $G$-cover $X \to \proj^1_k$ 
tamely ramified of index $m$ above $s = \infty$ and totally ramified
above $s = 0$ ($s$ being
a coordinate on $\proj^1_k$), such that the formal completion of $X \to \proj^1_k$ above
$0$ yields $k[[z]]/k[[s]]$.  

{\bf Proof of Proposition \ref{Pmaindef}:} 
The proof is essentially the same as the proof of \cite[Proposition
1.11]{Ob:go}, which itself is adapted from \cite{Po:oc}.  We
include it here for completeness.

Let $Y \to W = \proj^1$ be the Harbater-Katz-Gabber cover
associated to $k[[z]]/k[[s]]$,
let $\mc{Y} \to \mc{W}$ be the $A_4$-cover over $k[[\varpi]]$ guaranteed by Proposition
\ref{Pcharpglobal}, let $\ol{\mc{Y}} \to \ol{\mc{W}}$ be its base
change to the integral closure of $k[[\varpi]]$ in $\ol{k((\varpi))}$, and let $\ol{\mc{Y}_{\eta}} \to \ol{\mc{W}_{\eta}}$ be
the generic fiber of $\mc{Y} \to \mc{W}$. 
Recall that we assume that \emph{every} local $A_4$-extension
with break $\leq \nu - 6$ lifts to characteristic zero.  
Furthermore, by \cite{Pa:rc} and the theory of tame ramification,
every abelian extension of $k[[s]]$ (and thus of $\ol{k((\varpi))}((s))$)
with
Galois group a proper subgroup of $A_4$ lifts to characteristic zero.
So the local-global principle tells us that $\ol{\mc{Y}_{\eta}} \to \ol{\mc{W}_{\eta}}$ 
lifts to a cover $\mc{Y}_{\mc{O}_1} \to \mc{W}_{\mc{O}_1}$ over some
characteristic zero complete discrete valuation ring $\mc{O}_1$ with residue field $\ol{k((\varpi))}$.  
Then, \cite[Lemma 4.3]{Po:oc} shows that we can ``glue'' the covers
$\ol{\mc{Y}} \to \ol{\mc{W}}$ and $\mc{Y}_{\mc{O}_1} \to \mc{W}_{\mc{O}_1}$
along the generic fiber of the former and the special fiber of the
latter, in order to get a cover defined over a \emph{rank two} characteristic zero valuation ring $\mc{O}$ with
residue field $k$ lifting $Y \to W$ (cf.\ \cite[p.\ 319, second paragraph]{Po:oc}).  Note that this process works starting with any
$A_4$-extension of $k[[s]]$ with
break $\nu$, and that such extensions can be parameterized by some affine space $\aff^N$ 
(with one coordinate corresponding to each possible coefficient in an
entry of an Artin-Schreier generator in standard form).  

To conclude, we remark that \cite[Proposition 4.7]{Po:oc} and its setup carry through exactly in our situation, with our $\aff^N$ playing the role of 
$\aff^{|\mathbf{\iota}|}$ in \cite{Po:oc}.  Indeed, we have that the analog of $\Sigma_{\mathbf{\iota}}$ in that proposition contains all closed points,
by the paragraph above.  Thus we can in fact
lift $Y \to W$ over a \emph{discrete} characteristic zero valuation ring.  Applying the easy direction of the local-global principle, we obtain a lift of $k[[z]]/k[[s]]$.  This 
concludes the proof of Proposition \ref{Pmaindef}. \qed

\section{The form of a lift}\label{Sadjust}

We start by reviewing lifts of $\ints/2$-extensions of $k[[t]]$.  The
following lemma is well-known, but difficult to cite directly from the
literature.  We provide a proof.

\begin{lemma}\label{LZ2lifts}
Let $k((u))/k((t))$ be a $\ints/2$-extension with Artin-Schreier
generator $a \in t^{-1}k[t^{-1}]$ in standard form and ramification break $\nu$.  Let $A$
be a lift of $a$ to $T^{-1}R[T^{-1}]$ of degree $\nu$.  If $\Phi
\in 1 + T^{-1}\mf{m}[T^{-1}]$ has degree $\nu$ or $\nu + 1$ and
satisfies 
$$
\Phi = H^2 + 4A + o(4)
$$
for some $H \in 1 +
T^{-1}\mf{m}[T^{-1}]$, then the normalization of $R[[T]]$ in
$M := \Frac(R[[T]])[\sqrt{\Phi}]$ is a lift of $k[[u]]/k[[t]]$ to
characteristic zero.  Furthermore, $\Phi$ has simple roots.
\end{lemma}

\begin{proof}
The extension $k((u))/k((t))$ is given by adjoining an element $y$
such that $y^2 - y = a$.  Making a substitution $\sqrt{\Phi} = H + 2Y$,
the expression for $\Phi$ given in the lemma yields
$$H^2 + 4HY + 4Y^2 = H^2 + 4A + o(4),$$
or $Y^2 - Y = A + o(1)$.  Thus we see that
the normalization of $R[[T]]_{(\pi)}$ in $M$
gives $k((u))/k((t))$ upon reduction modulo $\pi$.  By Serre's
different formula (\cite[IV, Proposition 4]{Se:lf}), the degree of the
different of $k[[u]]/k[[t]]$ is $\nu + 1$.  On the other hand, the
normalization of $R[[T]] \otimes_R \Frac(R)$ in $M$ is branched at at
most $\nu + 1$ maximal ideals, corresponding to the roots of $\Phi$ and also $0$
if $\Phi$ has degree $\nu$.  Since this is a tamely ramified
$\ints/2$-extension, the degree of its different is at most $\nu + 1$.  By \cite[I,
3.4]{GM:lg}, the degree of the different is exactly $\nu + 1$ and the
normalization of $R[[T]]$ in $M$ is a lift of $k[[u]]/k[[t]]$.
This also shows that the roots of $\Phi$ are all simple.
\end{proof}

For Proposition \ref{Pliftform} below, recall that $s = t^3$ and $S = T^3$.
\begin{prop}\label{Pliftform}
Let $k[[z]]/k[[s]]$ be a local $A_4$-extension with break $\nu$ given
rise to by $a \in t^{-1}k[t^{-1}]$ in standard form.  If $F(T^{-1})$ and $H(T^{-1})$ are in $1
+ T^{-1}\mf{m}[T^{-1}]$ such that $F$ has degree $(\nu + 1)/2$ and 
$$F(\zeta_3T^{-1})F(\zeta_3^2T^{-1}) = H^2 + 4A + o(4),$$ where $A$ is a
lift of $a$ to $T^{-1}R[T^{-1}]$ of degree $\nu$, then the normalization of $R[[S]]$
in the $A_4$-extension of $\Frac(R[[S]])$ given rise to by
$F(\zeta_3T^{-1})F(\zeta_3^2T^{-1})$ is a lift of $k[[z]]/k[[s]]$ to
characteristic zero.
\end{prop}

\begin{proof}
Let the local $\ints/2$-extension $k[[u]]/k[[t]]$ be given by normalizing
$k[[t]]$ in the Artin-Schreier $\ints/2$-extension of $k((t))$ given
by $a$.  Let $L = \Frac(R[[T]])$.  By Lemma \ref{LZ2lifts}, normalizing $R[[T]]$ in the degree $2$ Kummer
extension $M/L$ given by some polynomial $\Phi \in 1 + T^{-1}\mf{m}[T^{-1}]$
of degree $\nu + 1$ in $T^{-1}$ such that $\Phi = H^2 + 4A + o(4)$ with $A$ as in the
proposition gives a lift of $k[[u]]/k[[t]]$ to characteristic zero,
and such a $\Phi$ has simple roots.

Let $\sigma$ generate $\Gal(L/\Frac(R[[S]]))$ (and also, by abuse of
notation, $\Gal(k((t))/k((s)))$).   Write $\Phi = F(\zeta_3T^{-1})F(\zeta_3^2T^{-1})$ for
some polynomial $F \in 1 + T^{-1}\mf{m}[T^{-1}]$ of degree $(\nu +
1)/2$ as in the proposition.  Then $\Phi$ has simple roots, and thus $F(T^{-1})$, $F(\zeta_3T^{-1})$,
and $F(\zeta_3^2T^{-1})$ have pairwise disjoint simple roots.
Consequently, the $\FF_2$-subspace of $L^{\times}/(L^{\times})^2$ generated by
$\Phi$, $\sigma(\Phi)$, and $\sigma^2(\Phi)$ has dimension $2$.  By Proposition \ref{PA4char0}, this is equivalent to the
Galois closure $N$ (over $\Frac(R[[S]])$) of $M$ having Galois group $A_4$.

Let $k((u'))/k((t))$ be the Artin-Schreier extension
given by $\sigma(a)$.  Clearly, the normalization of $R[[T]]$ in
$\Frac(R[[T]])[\sqrt{\sigma(\Phi)}]$ is a lift of $k[[u']]/k[[t]]$.
Note that $k[[z]]$ is the normalization of $k[[t]]$ in the compositum
of $k((u))$ and $k((u'))$.  Analogously, $N := \Frac(R[[T]])(\sqrt{\Phi},
\sqrt{\sigma(\Phi)})$ is the $A_4$-extension given rise to by $\Phi$. 
Now, $\Phi$ and $\sigma(\Phi)$ have exactly $(\nu+1)/2$ zeroes in
common.  Thus \cite[I, Theorem 5.1]{GM:lg} shows that the normalization of
$R[[T]]$ in $N$ is a lift of the Klein four extension $k[[z]]/k[[t]]$ (and is isomorphic to
$R[[Z]]/R[[T]]$ for $Z$ reducing to $z$).   
We conclude that $R[[Z]]/R[[S]]$ is a lift of $k[[z]]/k[[s]]$.

\end{proof}

\section{Proof of Theorem \ref{Tmain}}\label{Sbase}

In this section, let $k[[z]]/k[[s]]$ be a local $A_4$-extension given
rise to by $a \in t^{-1}k[t^{-1}]$ in standard form.  Recall that $\deg(a) = \nu$, where $\nu$
is the break in $k[[z]]/k[[s]]$.  We will prove that $k[[z]]/k[[s]]$
lifts to characteristic zero by strong induction on $\nu$.  

\begin{prop}\label{Pbase1}
If $\nu = 1$, then $k[[z]]/k[[s]]$ lifts to characteristic zero.
\end{prop}

\begin{proof}
Since $\nu = 1$, we have $a = \ol{c}_1t^{-1}$, with
$\ol{c}_1 \in k$.  By Proposition \ref{Pliftform}, it suffices to find $F(T^{-1})$ and $H(T^{-1})$ in $1 + T^{-1}R[T^{-1}]$
such that $F$ has degree $1$ and
$$F(\zeta_3T^{-1})F(\zeta_3^2T^{-1}) = H^2 + 4c_1T^{-1} + o(4),$$ where
$c_1$ is a lift of $\ol{c}_1$ to $R$.  This is accomplished by taking $H =
1$ and $F =  1 - 4c_1T^{-1}$.
\end{proof}

\begin{prop}\label{Pbase5}
If $\nu = 5$, then $k[[z]]/k[[s]]$ lifts to characteristic zero.
\end{prop}

\begin{proof}
Since $\nu = 5$, we have $a = \ol{c}_1t^{-1}
+ \ol{c}_5t^{-5}$, with
$\ol{c}_1, \ol{c}_5 \in k$.  By Proposition \ref{Pliftform}, it suffices to find $F(T^{-1})$ and $H(T^{-1})$ in $1 + T^{-1}R[T^{-1}]$
such that $F$ has degree $3$ and
$$F(\zeta_3T^{-1})F(\zeta_3^2T^{-1}) = H^2 + 4c_1T^{-1} + 4c_5T^{-5} +
o(4),$$ where each $c_i$ is a lift of $\ol{c}_i$ to $R$. 

Let $b \in R$ be any element such that $v(b) = 2/5$.  
Write $$F(T^{-1}) = 1 + a_1T^{-1} + a_2T^{-2} + a_3T^{-3},$$ where
$$a_1 = -2b - 4c_1, \ \ a_2 = b^2, \ \ a_3 = -4c_5/b^2.$$  Note that
$v(a_1) = 7/5$, $v(a_2) = 4/5$, and $v(a_3) = 6/5$. 
Then 
\begin{eqnarray*}
F(\zeta_3T^{-1})F(\zeta_3^2T^{-1}) &=& 1 - a_1T^{-1} - a_2T^{-2} +
                                       a_2^2T^{-4} - a_2a_3T^{-5} +
                                       o(4) \\
&=& 1 + (4c_1 + 2b)T^{-1} - b^2T^{-2} + b^4T^{-4} + 4c_5T^{-5} + o(4)
  \\
&=& (1 + bT^{-1} + b^2T^{-2})^2 + 4c_1T^{-1} + 4c_5T^{-5} + o(4).
\end{eqnarray*}
We conclude by taking $H = 1 + bT^{-1} + b^2T^{-2}$.

\end{proof}

{\bf Proof of Theorem \ref{Tmain}:}  We use strong induction on the break $\nu$ of $k[[z]]/k[[s]]$, which
only takes values congruent to $1$ or $5$ modulo $6$ (Proposition
\ref{P1or5}).  The base cases $\nu = 1$ and $\nu = 5$ are Propositions
\ref{Pbase1} and \ref{Pbase5}, respectively.  The induction step is
Proposition \ref{Pmaindef}. \qed

\begin{remark}\label{Rpop}
Florian Pop has informed the author of his own proof, which uses much
the same method.  In place of the deformation in
Proposition \ref{Pcharpoort}, he uses one for which it is slightly more difficult
to verify that it yields an $A_4$-extension, but which immediately reduces Theorem
\ref{Tmain} to the case $\nu = 1$ (eliminating the need for
Proposition \ref{Pbase5}).
\end{remark}

\begin{question}
Given $k$, does there exist a particular DVR $R$ in characteristic zero such that
all local $A_4$-extensions over $k$
lift over $R$?  This is known for local $G$-extensions in
characteristic $p$ where $G$ is cyclic with $v_p(|G|) \leq 2$
(see \cite{GM:lg}, where it is shown that $W(k)[\zeta_{p^2}]$ works).
Since our proof is rather inexplicit, this question remains open for $A_4$.
\end{question}

\bibliographystyle{alpha}
\bibliography{main}

\begin{thebibliography}{CGH11}

\bibitem[Ber98]{Be:ol}
Jos{\'e} Bertin.
\newblock Obstructions locales au rel\`evement de rev\^etements galoisiens de
  courbes lisses.
\newblock {\em C. R. Acad. Sci. Paris S\'er. I Math.}, 326(1):55--58, 1998.

\bibitem[BM00]{BM:df}
Jos{\'e} Bertin and Ariane M{\'e}zard.
\newblock D\'eformations formelles des rev\^etements sauvagement ramifi\'es de
  courbes alg\'ebriques.
\newblock {\em Invent. Math.}, 141(1):195--238, 2000.

\bibitem[BW06]{BW:ll}
Irene~I. Bouw and Stefan Wewers.
\newblock The local lifting problem for dihedral groups.
\newblock {\em Duke Math. J.}, 134(3):421--452, 2006.

\bibitem[BW09]{BW:ac}
Louis~Hugo Brewis and Stefan Wewers.
\newblock Artin characters, {H}urwitz trees and the lifting problem.
\newblock {\em Math. Ann.}, 345(3):711--730, 2009.

\bibitem[CGH08]{CGH:og}
Ted Chinburg, Robert Guralnick, and David Harbater.
\newblock Oort groups and lifting problems.
\newblock {\em Compos. Math.}, 144(4):849--866, 2008.

\bibitem[CGH11]{CGH:ll}
Ted Chinburg, Robert Guralnick, and David Harbater.
\newblock The local lifting problem for actions of finite groups on curves.
\newblock {\em Ann. Sci. \'Ec. Norm. Sup\'er. (4)}, 44(4):537--605, 2011.

\bibitem[Gar96]{Ga:pr}
Marco~A. Garuti.
\newblock Prolongement de rev\^etements galoisiens en g\'eom\'etrie rigide.
\newblock {\em Compos. Math.}, 104(3):305--331, 1996.

\bibitem[GM98]{GM:lg}
Barry Green and Michel Matignon.
\newblock Liftings of {G}alois covers of smooth curves.
\newblock {\em Compos. Math.}, 113(3):237--272, 1998.

\bibitem[Kat86]{Ka:lg}
Nicholas~M. Katz.
\newblock Local-to-global extensions of representations of fundamental groups.
\newblock {\em Ann. Inst. Fourier (Grenoble)}, 36(4):69--106, 1986.

\bibitem[Obu15]{Ob:go}
Andrew Obus.
\newblock A generalization of the {O}ort conjecture.
\newblock arxiv:1502.07623, 2015.

\bibitem[OW14]{OW:ce}
Andrew Obus and Stefan Wewers.
\newblock Cyclic extensions and the local lifting problem.
\newblock {\em Ann. of Math. (2)}, 180(1):233--284, 2014.

\bibitem[Pag02]{Pa:rc}
Guillaume Pagot.
\newblock Rel\`{e}vement en caract\'{e}ristique z\'{e}ro d'actions de groupes
  ab\'{e}liens de type $(p, \ldots, p)$.
\newblock Th\`{e}se, Universit\'{e} Bordeaux I, available at
  //http://www.math.u-bordeaux1.fr/~mmatigno/Pagot-These.pdf, 2002.

\bibitem[Pop14]{Po:oc}
Florian Pop.
\newblock The {O}ort conjecture on lifting covers of curves.
\newblock {\em Ann. of Math. (2)}, 180(1):285--322, 2014.

\bibitem[Ser68]{Se:lf}
Jean-Pierre Serre.
\newblock {\em Corps locaux}.
\newblock Hermann, Paris, 1968.
\newblock Deuxi{\`e}me {\'e}dition, Publications de l'Universit{\'e} de
  Nancago, No. VIII.

\bibitem[SGA03]{SGA1}
{\em Rev\^etements \'etales et groupe fondamental ({SGA} 1)}.
\newblock Documents Math\'ematiques (Paris) [Mathematical Documents (Paris)],
  3. Soci\'et\'e Math\'ematique de France, Paris, 2003.
\newblock S{\'e}minaire de g{\'e}om{\'e}trie alg{\'e}brique du Bois Marie
  1960--61. [Algebraic Geometry Seminar of Bois Marie 1960-61], Directed by A.
  Grothendieck, With two papers by M. Raynaud, Updated and annotated reprint of
  the 1971 original [Lecture Notes in Math., 224, Springer, Berlin; MR0354651
  (50 \#7129)].

\end{thebibliography}

\end{document}